\documentclass[english,leqno,10pt,a4paper,pdflatex]{article}

\usepackage[totalheight=25 true cm, totalwidth=17 true cm]{geometry}
\usepackage{pifont}
\usepackage{amsmath,amstext}
\usepackage{amsthm, amssymb}
\usepackage{upgreek}
\usepackage{hyperref}
\usepackage{bookmark}
\usepackage{srcltx}
\usepackage[english]{babel}
\usepackage[all]{xy}
\usepackage[utf8]{inputenc}
\usepackage[T1]{fontenc}

\usepackage{color}
\usepackage[color=orange,textwidth=1.5cm]{todonotes}
\usepackage{marginnote}
\definecolor{luciacolor}{rgb}{0.01,0.28,1.00}

\usepackage[continuous,page]{pagenote}
\makepagenote
\def\referee#1{}
\renewcommand*{\notesname}{Remarks for the referees}
\renewcommand*{\notedivision}{\section*{\notesname}
Below you will find some comments on the changes requested by the referees.
I haven't pointed out the minor changes concerning the spelling, but I have followed all the recommandations, of course.
I have numbered the three reports following the editors' numbering: I hope that each referee will recognize herself/himself.}

\newtheorem{thm}{Theorem}[section]
\newtheorem{cor}[thm]{Corollary}
\newtheorem{lemma}[thm]{Lemma}
\newtheorem{prop}[thm]{Proposition}

\theoremstyle{plain}
\newtheorem{defn}[thm]{Definition}

\theoremstyle{definition}
\newtheorem{rmk}[thm]{Remark}

\newtheorem{exa}[thm]{Example}

\newtheorem{assumption}[thm]{Assumption}

\def\$$endproof{\eqno{\qedhere}$$\end{proof}}

\numberwithin{equation}{section}

\def\beq{\begin{equation}}
\def\eeq{\end{equation}}

\def\G{{\mathbb G}}
\def\Z{{\mathbb Z}}
\def\Q{{\mathbb Q}}

\def\C{{\mathbb C}}

\def\F{{\mathbb F}}
\newcommand{\QQ}{\mathbb{Q}}
\def\CC{{\mathbb{C}}}

\def\cL{{\mathcal{ L}}}

\def\a{\alpha}

\def\sg{\sigma}

\def\Ga{\Gamma}
\def\la{\lambda}
\def\La{\Lambda}

\def\GL{\hbox{\rm GL}}
\def\SL{\hbox{\rm SL}}
\def \Hom{{\rm Hom}}
\def\Gal{\mathop{\rm Gal}}
\def\Aut{\mathop{\rm Aut}}
\def\l{\left}
\def\r{\right}
\def\PV{Picard-Vessiot}
\def\wtilde{\widetilde}
\def\l{\left}
\def\r{\right}
\def\[[{\l[\l[}
\def\]]{\r]\r]}
\def\p{\prime}
\def\ds{\displaystyle}

\makeindex

\title{Difference Galois theory for the ``applied'' mathematician}
\author{Lucia Di Vizio}

\begin{document}

\bibliographystyle{amsalpha}

\maketitle

\begin{abstract}
The lecture notes below correspond to the course given by the author in occasion of the
VIASM school on Number Theory (18-24 June 2018, Hanoi).
We have chosen to omit the proofs that are already presented in details in many references in the literature,
although they were explained during the lectures, and we have devoted more space to statements useful in the applications.
The applications concern many different mathematical settings, where linear difference equations naturally arise.
We cite in particular the case of Drinfeld modules, which is considered in \cite{pellarin2019carlitz} and
\cite{tavares2019stark}.
\end{abstract}

\setcounter{tocdepth}{1}
\tableofcontents

\section{Introduction}

The initial data of classical Galois theory are a field $K$, let's say of characteristic $0$,
and an irreducible polynomial $P\in K[x]$, with coefficients in $K$.
Then \emph{the} minimal field extension $L$ of $K$ containing a \emph{full set} of roots of
$P$ is constructed and one defines the Galois group $G$ of $L/K$, namely the group of the field automorphisms
of $L$ that fix the elements of $K$.
The group $G$ is finite and acts on $L$ by shuffling the roots of $P$.
The idea behind this construction is that the structure of the group $G$
should reveal hidden algebraic relations among the roots of $P$, other than the evident relations
given by $P$ itself.
\par
The same kind of philosophy applied to functional equations has been the starting
point of differential Galois theory, first, and difference Galois theory, later.
The references are numerous and it is almost impossible to list them all.
We refer to \cite{vanderPutSinger:DifferentialGaloisTheory}, for the differential case, and to  \cite{vanderPutSinger:DifferenceGaloisTheory} and \cite{hardouin-Sauloy-Singer-SantaMarta},
for the difference case.
\par
Let $K$ be a base field and $\tau$ be an endomorphism of $K$. We consider a linear difference equation $\tau(Y)=AY$, where
$A$ is a square invertible matrix with coefficients in $K$ and $Y$ is a square matrix of unknowns.
The usual approach in Galois theory of difference equations is to construct an abstract $K$-algebra $L/K$ containing the entries of a fundamental (i.e., invertible)
matrix of solutions, under the assumption that the characteristic of $K$ is zero, that the field of constants $k$ is algebraically closed and that $\tau$ is surjective.
The idea is that, in order to study the properties of the solutions a difference equation,
it is not ``important'' to solve it,
but only to understand the structure of its Galois group.
However, in applications, it usually happens that a set of solutions is given in some specific $K$-algebra:
When that's the case, it is not always easy to understand which properties transfer from the abstract solutions to the ones
that we have found.
\par
In this paper we suppose that we have a fundamental solution is some field $L/K$, equipped with an extension of $\tau$.
The divergence between the classical approach and the apparently more pedestrian approach that we are considering here, starts immediately:
Indeed in general we cannot assume that the field $L$ exists and only the existence
of a pseudo field is ensured (see Remark \ref{rmk:Picard-Vessiot} below).
In \cite{ChatzidakisHardouinSinger}, the authors reconcile these two points of view and the point of view of model theory:
They construct a group using given solutions
in a specific algebra and compare it with the group constructed in \cite{vanderPutSinger:DifferenceGaloisTheory}, but
they assume that $\tau$ is an automorphism
and the statements on the comparison of the Galois groups are not easy to apply in other settings.
\par
More recently, a very general abstract approach has been considered by A. Ovchinnikov and M. Wibmer in
\cite[\S2.2]{OvchinnikovWibmerSGaloisTheoryofLinearDifferenceEquations}, where,
in contrast with the more classical references,
the authors do not make any assumption on the characteristic of the field, they do not require that the endomorphism is surjective, and they do not assume that
the constants are algebraically closed. They do not even assume that $K$ is a field, but only that $K$ is a pseudo field.
\par
M. Papanikolas in \cite[\S4]{Papanikolas:TannakianDuality} chooses a framework which is in between the two examples above:
He works on a field $K$ equipped with an automorphisms $\tau$, but the characteristic can be positive and the field of constants is not necessarily algebraically closed.
Moreover he supposes that he already has a fundamental solution in a field extension of $K$.
We will consider the same setting as Papanikolas, apart from the fact that we only ask that $\tau$ is an endomorphism.
This seems to be a reasonable framework for many applications.
For the proofs, we usually refer to \cite{OvchinnikovWibmerSGaloisTheoryofLinearDifferenceEquations},
which is the reference with more general assumptions.
Notice that Papanikolas has a more geometric approach while Ovchinnikov and Wibmer prefer algebraic arguments.
\par
Finally we point out that we assume that the characteristic is zero in \S\ref{subsec:Dfinite} and \S\ref{sec:applications} and
that $\tau$ is an automorphism in \S\ref{sec:GaloisCorrespondenceNormal}.

\paragraph*{Remarks on the content and the organization of the text below.}
The text below is meant to be a guide to the existing literature.
From this perspective, I will give references for the proofs,
rather than writing a self-contained exposition.
I'm addressing with particular attention readers that need to apply Galois theory of difference equations,
therefore a large space is devoted to statements that may be useful in applications.
\par
The exposition is divided in three parts. The first part quickly explains the fundamentals results and ideas of difference Galois theory.
In order to be precise and correct I have been obliged to use some more sophisticated tools.
The second part, devoted to applications, is more accessible and applies the statements of the first part as black boxes.
We conclude with a last paragraph on the role of normal subgroups in the Galois correspondence.

\paragraph*{Acknowledgements} I'm indebted to the organizers of the VIASM school on Number Theory, in particular to Bruno Anglès and Tuan Ngo Dac
for their invitation both to give the course and to write the this paper.
I'd like to thank the participants of the \emph{Groupe de travail autour des marches dans le quart de plan}, who have endured
several talks on this topic, that have influenced the text below.
In particular, I'm grateful to Alin Bostan, Frédéric Chyzac and Marni Mishna for their
remarks and their attentive reading of various drafts of this survey and to the anonymous referees for the many useful and constructive comments.

\section{Glossary of difference algebra}

We give here a short glossary of terminology in difference algebra.
Classical references are \cite{Cohn-DifferenceAlgebra} and \cite{LevinDifferenceAlgebra}.

\medskip
We consider a field $F$ equipped with an endomorphism $\tau$.
We will call the pair $(F,\tau)$ a $\tau$-field or a difference field, when the reference to $\tau$ is clear from the context.
The set $k=\{f\in F:\tau(f)=f\}$, also denoted $F^\tau$, is naturally a field and is  called the field of constants of $F$.
All along this exposition, we will assume that $\tau$ is non-periodic on $F$
(i.e. there exists $f\in F$ such that for any $n\in\Z$ we have $\tau^n f\neq f$)
and we won't assume that $k$ is algebraically closed.
The example below will be our playground until the end of the paper.

\begin{exa}\label{exa:basic}
We consider $k=\CC$, $F=k(x)$ and $\tau$ a non-periodic homography acting on $x$, so that $\tau(f(x))=f(\tau(x))$.
Supposing that $\tau$ has one or two fixed points, we can assume without loss of generality that
$\tau(x)=x+1$ or that $\tau(x)=qx$, for some $q\in\C\smallsetminus\{0,1\}$, not a root of unity.
\end{exa}

We will add the prefix $\tau$ to the usual terminology in commutative algebra, to signify the invariance with respect to $\tau$.
For instance:
\begin{itemize}
  \item A $\tau$-subfield $K$ of a $\tau$-field $F$ is a subfield of $F$ such that $\tau$ induces an endomorphism of $K$, i.e., such that $\tau(K)\subset K$.
  \item A $\tau$-$K$-algebra is a $K$-algebra equipped with an endomorphism extending $\tau$ (which we still call $\tau$ for simplicity).
  \item An ideal $I$ of a $\tau$-$K$-algebra $R$ is a $\tau$-ideal if $\tau(I)\subset I$.
  \item A $\tau$-$K$-algebra $R$ is $\tau$-simple if its only $\tau$-ideals are $R$ and $0$.
\end{itemize}

\begin{exa}\label{exa:basic-example-difference-field}
For any $x\in F$ transcendental over $k$, such that
$\tau(x)\in k(x)$ and that for any $n\in\Z$ we have $\tau^n(x)\neq x$, we can consider the field
$K:=k(x)$, on which $\tau$ induces a non-periodic endomorphism.
Since we don't ask $\tau$ to be surjective, the situation is
a little bit more general than
Example \ref{exa:basic}. For instance $\tau$ can be the Mahler operator $\tau(x)=x^\kappa$, where $\kappa\geq 2$ is an integer.
\par
For any positive integer $n$, we can consider $k_n=F^{\tau^n}$ and the
difference fields $(K_n:=k_n(x),\tau)$
and $(K_n:=k_n(x),\tau^n)$.
\end{exa}

Sometimes it is useful to
consider $(k,id)$, where $id$ stands for the identity map, as a difference field, therefore we will call it a trivial $\tau$-field. In general we will say that a $k$-algebra is equipped with a trivial action of $\tau$,
when $\tau$ acts on it as the identity.
For further reference we state the following lemma:

\begin{lemma}\label{lemma:LinDisj}
Let $B$ be a $k$-algebra endowed with a trivial action of $\tau$,  $K$ a $\tau$-subfield of $F$ such that $K^\tau=k$ and
let $R\subset F$ be a $\tau$-$K$-algebra. Then $R\otimes_k B$ has a natural structure of $\tau$-$K$-algebra defined by
\beq\label{eq:LinDis}
\tau(r\otimes b)=\tau(r)\otimes b, \hbox{~for any $r\in R$ and $b\in B$.}
\eeq
Moreover $R\otimes_kB\hookrightarrow F\otimes_k B$ and
$(R\otimes_k B)^\tau=k\otimes_k B\cong B$.
\end{lemma}

\begin{rmk}
We have chosen to detail the proof, even if the lemma above is included in many references and could partially be proved invoking
the flatness of $F$ over $k$. The argument is indeed an instance of a classical way of reasoning in difference algebra and is useful in many situations.
\end{rmk}

\begin{proof}
Clearly Eq. \eqref{eq:LinDis} defines a ring endomorphism, that is the tensor product of $\tau$ and the identity in the category of rings.
Since $R\subset F$, we have a natural map of $\tau$-$K$-algebras $R\otimes_k B\to F\otimes_k B$.
We want to prove that this map is injective. By absurdum,
we suppose that the kernel is non-trivial and therefore
we choose a non-zero element $\sum_{i=1}^n  r_i\otimes b_i\in R\otimes_k B$ in the kernel
such that $n$ is minimal, i.e., we suppose that there exists no element of $R\otimes_k B$ in the kernel,
that can be written as a sum of less than $n$ elements of the form $r\otimes b\in R\otimes_k B$.
This implies in particular that the $r_i$'s are linearly independent over $k$ and that
all the $r_i$'s and $b_i$'s are non-zero.
Since $F$ is a field, in $F\otimes_k B$ we can multiply by $r_n^{-1}\otimes 1$, hence we have:
$1\otimes b_n+\sum_{i=1}^{n-1} r_ir_n^{-1}\otimes b_i=0$ in $F\otimes_k B$.
We conclude that the image of
$$
\sum_{i=1}^{n-1}\l(\tau(r_ir_n^{-1})-r_ir_n^{-1}\r)\otimes b_i\in R\otimes_k B
$$
in $F\otimes_k B$ is zero. The minimality of $n$, together with the fact that $b_i\neq 0$,
implies that $\tau(r_ir_n^{-1})-r_ir_n^{-1}=0$ and hence that
$r_ir_n^{-1}\in k$ for any $i=1,\dots,n-1$.
The linear independence of the $r_i$'s over $k$
implies that $n=1$ and hence that $r_1\otimes b_1=0$ in $F\otimes_k B$, with $r_1\neq 0$.
Since $F$ is a field, $r_1^{-1}\otimes 1\in F\otimes_k B$ and hence:
$$
1\otimes b_1=(r_1^{-1}\otimes 1)\cdot(r_1\otimes b_1)=0
$$
in $F\otimes_k B$, and we obtain the contradiction $b_1=0$.
This proves the injectivity.
\par
To conclude it is enough to prove that $(R\otimes_k B)^\tau=k\otimes_k B$.
First of all, notice that if $r\otimes b\in (R\otimes_k B)^\tau$, with $r\neq 0\neq b$, then $r\in k$.
By absurdum, let us suppose that there exists $\sum_{i=1}^n  r_i\otimes b_i\in (R\otimes_k B)^\tau\smallsetminus k\otimes_k B$,
such that no $r_i$ belongs to $k$ and $n$ is minimal.
Once again, the minimality of $n$ implies that the $r_i$'s and the $b_i$'s are linearly independent over $k$.
Moreover we know that necessarily $n\geq 2$.
We conclude that
$$
\sum_{i=1}^{n}\l(\tau(r_i)-r_i\r)\otimes b_i=\tau\l(\sum_{i=1}^n  r_i\otimes b_i\r)-\l(\sum_{i=1}^n  r_i\otimes b_i\r)=0.
$$
There are two possibility: either the $(\tau(r_i)-r_i)$'s are linearly dependent over $k$ or they are not.
If they are linearly dependent over $k$, we can find $\la_1,\dots,\la_n\in k$, not all zero, such that $\sum_{i=1}^n\la_i(\tau(r_i)-r_i)=0$.
We can suppose without loss of generality that $\la_n\neq 0$.
It implies that
$$
\sum_{i=1}^n\la_i r_i=\sum_{i=1}^n\la_i\tau(r_i)=\tau\l(\sum_{i=1}^n\la_i r_i\r),
$$
and hence that $c:=\sum_{i=1}^n\la_i r_i\in k$. Since the $r_i$'s are linearly independent over $k$, $c$ is not zero.
The $k$-linearity of the tensor product implies:
$$
\begin{array}{rcl}
\sum_{i=1}^n  r_i\otimes b_i
    &=&\sum_{i=1}^{n-1}  r_i\otimes b_i+\la_n^{-1}\l(c-\sum_{i=1}^{n-1}\la_i r_i\r)\otimes b_n\\~\\
    &=&\sum_{i=1}^{n-1}  r_i\otimes (b_i-\la_n^{-1}\la_i b_n)+\la_n^{-1}c\otimes b_n\in (R\otimes_k B)^\tau.
\end{array}
$$
Because $\la_n^{-1}c\otimes b_n\in k\otimes_k B\subset (R\otimes_k B)^\tau$, we conclude that $\sum_{i=1}^{n-1}  r_i\otimes (b_i-\la_n^{-1}\la_i b_n)\in (R\otimes_k B)^\tau$.
The minimality of $n$ implies that $b_i=\la_n^{-1}\la_ib_n$, for any $i=1,\dots,n-1$.
Finally we obtain:
$$
\sum_{i=1}^n  r_i\otimes b_i=\l(\sum_{i=1}^n \la_n^{-1}\la_i r_i\r)\otimes b_n\in (R\otimes_k B)^\tau,
$$
and therefore that
$\sum_{i=1}^n  r_i\otimes b_i\in k\otimes_k B$, in contradiction with our choice of $\sum_{i=1}^n  r_i\otimes b_i$.
We still have to consider the case in which $(\tau(r_i)-r_i)$'s are linearly independent over $k$, with
$\sum_{i=1}^{n}\l(\tau(r_i)-r_i\r)\otimes b_i=0$ in $F\otimes_kB$, but we have seen in the first part of the proof that this cannot happen,
unless $\tau(r_i)=r_i$, which is against our assumptions.
This ends the proof of the whole lemma.
\end{proof}

\section{Picard-Vessiot rings }
\label{sec:PVrings}

We now consider a field $F$ with an endomorphism $\tau:F\to F$.
Our base field will be a $\tau$-subfield $K$ of $F$,
containing $k:=F^\tau$, which implies that $K^\tau=k$.
We do assume neither that $k$ is algebraically closed, nor that $\tau$ induces an automorphism of $K$,
but we assume that $\tau$ is non-periodic over $K$.
Moreover we assume that $F/k$ is a separable extension.

\begin{rmk}
In the classical theory, one usually assumes that $k$ is algebraically closed, which simplifies a little the theory,
although not in a fundamental way. The main difference comes from the fact that the Galois group, that we will define in the following section, is an algebraic group over the field $k$.
Therefore if $k$ is algebraically closed one can avoid a more
sophisticated geometric point of view and simply identify
the Galois group to a group of invertible matrices with coefficients in $k$.
This remark will become clearer in what follows. We will comment again on the consequences of the fact that $k$ is not algebraically closed.
\par
We also point out that the assumption $F^\tau=K^\tau=k$ is crucial,
otherwise we may end up introducing new meaningless solutions in the theory. For more details, see Example \ref{exa:constants}
and the proof of Proposition \ref{prop:transcendenceINH} below.
\end{rmk}

\medskip
We consider a linear difference system
\beq\label{eq:system}
\tau(Y) = AY,
\hbox{~where $A\in\GL_d(K)$},
\eeq
and we suppose that \emph{there exists a fundamental solution matrix $U\in\GL_d(F)$ of
\eqref{eq:system}.} Then the field $L:=K(U)\subset F$ is obviously stable by $\tau$.
We have made an abuse of notation that we will repeat frequently:
By $K(U)$ we mean the field generated over $K$ by the entries of $U$.
\par
There are two main situations that the readers, according to their background, could keep in mind as a guideline through the text below:

\begin{exa}
One can consider the following two classical situations:
\begin{enumerate}
\item
	$F$ is the field of meromorphic functions over $\C$ in the variable $x$ and $\tau: f(x)\mapsto f(x+1)$. Then $k$ is the field of
    meromorphic $1$-periodic functions over $\C$.
\item
	$F$ is the field of meromorphic functions over $\C^\ast$ in the variable $x$ and $\tau: f(x)\mapsto f(qx)$, for a fixed complex number
	$q\in\C\smallsetminus\{0,1,\hbox{~roots of unity}\}$. Here $k$ is the field of meromorphic $q$-elliptic functions over $\C^\ast$.
\end{enumerate}
In both cases, we can chose $K$ to be any $\tau$-subfield of $F$, containing $k$.
A typical choice for $K$ is $k(x)$, which is the point of view taken in \cite{ChatzidakisHardouinSinger}, as far as $q$-difference equations is
concerned.
\par
Let as consider a linear system $\tau(\vec{y})=A\vec{y}$ with coefficients in $K$, of the form \eqref{eq:system}.
In the settings above, plus the assumption $|q|\neq 1$ in the $q$-difference case, Praagman proves that $\tau(\vec{y})=A\vec{y}$
has a fundamental matrix of solutions with coefficients in $F$.
See \cite[Theorem 1 and Theorem 3]{PraagmanFundamentalSolutionsForMeromorphicLinearDifferenceEquations}.
This does not mean that all possible solutions are meromorphic.
Indeed it is enough to multiply a solution matrix by a matrix whose entries are functions with some essential singularities
and that are constant with respect to $\tau$.
\end{exa}

\begin{exa}
In \cite[\S4.1]{Papanikolas:TannakianDuality} the triple $(k,K,F)$ is called a $\tau$-admissible triple.
He is specifically interested in the study of $t$-motifs and, hence, of the associated triple
$(\F_q(t),\mathbb{K},\mathbb{L})$, defined as follows \cite[\S2.1]{Papanikolas:TannakianDuality}:
\begin{itemize}
\item
$\F_q(t)$ is the field of rational functions in the variable $t$ and with coefficients in the field $\F_q$ with $q$ elements, where
$q$ is an integer power of a prime $p$;
\item
$\mathbb{K}$ is the smallest algebraically closed and complete extension of a field of rational functions $\F_q(\theta)$,
with respect to the $\theta^{-1}$-adic valuation.
\item
$\mathbb{L}$ is the field of fractions of the ring of all power series in $\mathbb{K}[[t]]$ convergent on the closed unit disk for the
$\theta^{-1}$-adic valuation.
\end{itemize}
The automorphism $\tau$ is the inverse of the Frobenius morphism on the algebraic closure of $\F_q$,
sends $\theta$ to $\theta^{1/q}$ , and is defined on $\mathbb{K}[[t]]$ as
$$
\tau\l(\sum_{n\geq 0}a_nt^n\r)=\sum_{n\geq 0}a_n^{1/q}t^n.
$$
Finally, it extends to $\mathbb{L}$ by multiplicativity.
Then $\mathbb{L}^{\tau}=\mathbb{K}^{\tau}=\F_q(t)$, as proved in Lemma 3.3.2 in \emph{loc.cit.}.
\end{exa}

We start mentioning a technical but fundamental property of the $\tau$-$K$-algebra generated by the entries of the solution matrices of \eqref{eq:system},
which allows to reconnect our point of view with the more classical approach of the \PV\ theory.

\begin{prop}\label{prop;tausemplicity}
In the notation above, let $U\in\GL_d(F)$ verify $\tau(U)=AU$.
Then $R=K\l[U,\det U^{-1}\r]\subset F$ is $\tau$-simple.
\end{prop}

\begin{proof}
The statement is a special case of
\cite[Proposition 2.14]{OvchinnikovWibmerSGaloisTheoryofLinearDifferenceEquations}, as one can see from Definitions 2.2 and 2.4 in \emph{loc.cit.}
\end{proof}

At this point, the reader should pay attention to the terminology in the literature.
The ring $R$ in the proposition above and its field of fractions $L$ are a \emph{weak} Picard-Vessiot ring and a
\emph{weak} Picard-Vessiot field, respectively, according to \cite[Definition 2.1]{ChatzidakisHardouinSinger}.
Indeed it follows immediately from their definition that $R^\tau=L^\tau=k$.
Proposition \ref{prop;tausemplicity} shows that $R$ and $L$ are respectively a Picard-Vessiot ring and a Picard-Vessiot field,
following also the definition \cite[Defintion 2.12]{OvchinnikovWibmerSGaloisTheoryofLinearDifferenceEquations}.
For the purpose of this paper, we will use a terminology in between \cite{ChatzidakisHardouinSinger} and
\cite{OvchinnikovWibmerSGaloisTheoryofLinearDifferenceEquations}, knowing that $R$ and $L$ satisfy
the definitions in both the cited references, and that, therefore,
the results in both references apply here.

\begin{defn}\label{defn:PV}
A $\tau$-simple $\tau$-$K$-algebra
$R$ is called
a \PV\ ring (for \eqref{eq:system}) if
there exists $V\in\GL_d(R)$ such that $\tau(V)=AV$ and $R=K\l[V,\det V^{-1}\r]$.
\par
A difference field $L$ is called
a \PV\ field over $K$ (for \eqref{eq:system}) if
$L^\tau=k$ and
$L=K(V)$ for a $V\in\GL_d(L)$ such that $\tau(V)=AV$.
We will call $L/K$ a \PV\ extension.
\end{defn}

\begin{rmk}\label{rmk:Picard-Vessiot}
If we do not have a field where to find enough solutions of our equation, we have to construct an abstract Picard-Vessiot extension.
To do so, one considers the ring of polynomials in the $d^2$ variables $X=(x_{i,j})$ with coefficients in $K$.
Inverting $\det X$ and setting $\tau(X)=AX$, we obtain a ring $K[X,\det X^{-1}]$ with an endomorphism $\tau$.
Any of its quotients by a maximal $\tau$-invariant ideal is a Picard-Vessiot ring of  $\tau(\vec{y})=A\vec{y}$ over $K$.
It is important to notice that the ring $R$ does not need to be a domain.
It can be written has a direct sum $R_1\oplus\dots\oplus R_r$, such that:
$R_i=e_i R$, for some $e_i\in R$ such that $e_i^2=e_i$;
$R_i$ is a domain; there exists a permutation $\sg$ of $\{1,\dots,r\}$ such that $\tau(R_i)\subset R_{\sg(i)}$.
See \cite[Cor.~1.16]{vanderPutSinger:DifferenceGaloisTheory}.
Its total field of fractions is a pseudo field, that is the
tensor product of the fraction field of each $R_i$.
For a precise definition of pseudo field, see for instance
\cite[Definition 2.2]{OvchinnikovWibmerSGaloisTheoryofLinearDifferenceEquations}.
\end{rmk}

The following example shows the importance of the assumption
$L^\tau=k$ in the definition of \PV\ field.

\begin{exa}\label{exa:constants}
Let us consider the equation $\tau(y)=-y$, with $k$, $K$ and $F$ as above.
We suppose that there exists a $2$-dimensional $k$-vector space of solutions of $\tau(y)=-y$ in $F$, which coincides with $k_2:=F^{\tau^2}$.
The Picard-Vessiot field (contained in $F$) of $\tau(y)=-y$ over $K$ is $L:=K(k_2)$.
\par
We could also have considered a field of rational function $F(T)$ with coefficients in $F$ and in the variable $T$.
Since $T$ is transcendental, we can set $\tau(T)=-T$ and obtain an endomorphism of $F(T)$.
If we do not assume that the Picard-Vessiot field has the same field of constants than the base field $K$,
we see that $K(k_2)(T)$ is a Picard-Vessiot field for $\tau(y)=-y$, whose field of constants is $k_2(T^2)$.
Of course, the solution $T$ is somehow artificial and
the extension $K(k_2)(T)/K$ is much bigger (i.e. has many more automorphisms, see next section) than $K(k_2)/K$.
\end{exa}

The expected relations between \PV\ rings and \PV\ fields are verified. If $R, L\subset F$, the proof is actually straightforward.

\begin{cor}\emph{\cite[Proposition 2.15]{OvchinnikovWibmerSGaloisTheoryofLinearDifferenceEquations}}
\begin{trivlist}
\item
1. Let $R$ be domain which is a \PV\ ring. Then its field of fractions is a \PV\ field.
\item
2. Let $L$ be a \PV\ field for \eqref{eq:system} and let $U\in\GL_d(L)$ be a solution of \eqref{eq:system}.
Then $K[U,\det U^{-1}]$ is a \PV\ ring.
\end{trivlist}
\end{cor}

Notice that one can always compare two different Picard-Vessiot rings, up to an algebraic extension of
the field of constants:

\begin{prop}
\PV\ rings have the following uniqueness properties:
\begin{enumerate}
\item
Let $R_1$ and $R_2$ be two \PV\ rings for \eqref{eq:system}, both contained in $F$.
Then $R_1=R_2$.
\item
Let $R\subset F$ and $R^\p$ be two \PV\ rings for \eqref{eq:system}. (Notice that
we do not suppose that $R^\p\subset F$!) Then there exists an algebraic field extension $\wtilde k$ of $k$,
containing a copy of $k^\p:=(R^\p)^\tau$, such that $R\otimes_k\wtilde k$ is isomorphic to $R^\p\otimes_{k^\p}\wtilde k$ as a
$K\otimes_k\wtilde  k$-$\tau$-algebra.
\end{enumerate}
\end{prop}

\begin{proof}
The first assertion follows from the fact that any pair of fundamental solutions
$U,V\in\GL_d(F)$ verifies $\tau(U^{-1}V)=U^{-1}V$, i.e., $U^{-1}V\in\GL_d(k)$. In fact, this implies that
$K[U,\det U^{-1}]=K[V,\det V^{-1}]\subset F$.
For the second assertion,
see \cite[Theorem 2.16]{OvchinnikovWibmerSGaloisTheoryofLinearDifferenceEquations}.
Notice that the key-point of its proof is Lemma 2.13 in \emph{loc.cit.}, which ensure that $k^\p/k$ must be an
algebraic extension.
\end{proof}

We give an explicit
example to explain the necessity of extending the constants to $\wtilde k$ in the statement
above.

\begin{exa}
Let $k$ be the field of $1$-periodic meromorphic functions over $\C$ and let $K=k(x)$.
The Picard-Vessiot ring of the equation $\tau(y)=xy$ over $K$, contained in $F$, is
$R:=K[\Gamma(x),\Gamma(x)^{-1}]$, where $\Gamma(x)$ is the Euler Gamma function.
\par
Now let $f(x)$ be a $2$-periodic function, algebraic over $k$, but not meromorphic over $\C$.
It means that $f(x)$ lives on an analytic $2$-fold covering of $\C$ and has some branching points.
The function $c(x):=f(2x)$ is $1$-periodic but does not belong to $F$, and hence not to $k$.
The $K$-algebra $R^\p:=K[\wtilde\Gamma(x),\wtilde\Gamma(x)^{-1}]$, where $\wtilde\Gamma(x):=c(x)\Gamma(x)$,
is a Picard-Vessiot ring for $\tau(y)=xy$.
In the notation of the proposition above, we have $k^\p=k$ and $\wtilde k:=k(c(x))$.
Indeed, $\Gamma(x)\mapsto c(x)\Gamma(x)$ defines an automorphism from $R\otimes_k\wtilde k$ to $R^\p\otimes_k\wtilde k$
as $K\otimes_k\wtilde k$-algebras.
\end{exa}

We close the section with a couple of easy, yet crucial, examples:

\begin{exa}\label{exa: order1-PV}
Let $a$ be a non-zero element of $K$ and let us consider the rank-one equation $\tau(y)=ay$.
By assumption there exists a solution $z\in F$ verifying $\tau(z)=az$. Hence
$K[z,z^{-1}]$ is a \PV\ ring for $\tau(y)=ay$ and $K(z)$ is a \PV\ field.
Generically, $z$ is transcendental over $K$, but not always.
For instance, if $F$ is the field of meromorphic functions over $\C$ in the variable $x$,
$\tau(f(x))=f(x+1)$ for any $f\in F$, $K=k(x)$ and $a=-1$, we can take $z=\exp(\pi ix)\in F$.
In this case $K(z)$ is a finite extension of degree 2, since $\exp(\pi ix)^2=\exp(2\pi ix)$ is a
$1$-periodic function, belonging to $k$.
\end{exa}

\begin{exa}\label{exa:log}
Let $f\in K$ and let us consider the inhomogeneous equation
$\tau(y)=y+f$. Such an equation is equivalent to the matrix system
$$
\tau(Y)=\begin{pmatrix}
1 & f\\ 0 & 1
\end{pmatrix}Y,
$$
whose fundamental solution is given by $Y=\begin{pmatrix}1& z\\0 &1 \end{pmatrix}$,
where $z\in F$ verifies $\tau(z)=z+f$.
Since $Y^{-1}=\begin{pmatrix}1& -z\\0 &1 \end{pmatrix}$,
the \PV\ ring of $\tau(y)=y+f$ is $K[z]$ and the \PV\ field is $K(z)$.
\end{exa}

\section{The Galois group}

The Galois group of a difference system of the form \eqref{eq:system} is a linear algebraic group defined over the field of constants $k$.
As we have already pointed out,
since we have chosen not to assume that  $k$ is algebraically closed,
we cannot stick to a naive approach to linear algebraic groups as sets of matrices with entries in the base field,
but we have to use the point of view of group schemes.
For the reader's convenience we recall informally a minimal amount of definitions that are necessary in what follows.
They are contained in any classical reference on group schemes, for instance \cite{Waterhouse:IntrotoAffineGroupSchemes}.

\subsection{A short digression on group schemes}

A group scheme $G$
over the field $k$ is a covariant functor
from the category of $k$-algebras to the category of groups:
$$
\begin{array}{rccc}
 G : & \hbox{$k$-algebras} & \to & \hbox{Groups} \\
   & B & \mapsto & G(B)
\end{array}.
$$
An affine group scheme is a group scheme which is representable, i.e., there exists a $k$-algebra $k\l[G\r]$
such that the functor $G$ and $\Hom_k(k\l[G\r],-)$ are naturally isomorphic.
This implies, in particular, that $G(B)$ and $\Hom_k(k\l[G\r],B)$ are isomorphic as groups, for any $k$-algebra $B$.

\begin{exa}\label{exa:groupschemes}
For $k=\QQ$, we can look at $\GL_n$ as an affine group scheme over $\QQ$ in the following way:
$$
\begin{array}{rccc}
 \GL_{n,\QQ} : & \hbox{$\QQ$-algebras} & \to & \hbox{Groups} \\
   & B & \mapsto & \GL_n(B)
\end{array}.
$$
We recall that, for a general $\QQ$-algebra $B$, $GL_n(B)$ is the group of square matrices with coefficients in $B$,
whose determinant is an invertible element of $B$.
We have:
$$
\QQ\l[\GL_{n,\QQ}\r]=\frac{\QQ[t, x_{i,j},\, i,j=1,\dots, n]}{(t\det(x_{i,j})-1)}.
$$
Of course an analogue definition holds for the affine group scheme $\GL_{n, k}$, define over a generic field $k$.
For $n=1$, we obtain the multiplicative affine group scheme, for whom we will use the notation $\G_{m,k}$ rather then $\GL_{1,k}$.
The additive affine group scheme $\G_{a,k}$ is defined as follows:
$$
\begin{array}{rccc}
 \G_{a,k} : & \hbox{$k$-algebras} & \to & \hbox{Groups} \\
   & B & \mapsto & \l\{\begin{pmatrix}1 & b \\ 0 & 1\end{pmatrix}\vert\, b\in B\r\}.
\end{array}.
$$
We have:
$$
k\l[\G_{a,k}\r]=\frac{k\l[\GL_{2,k}\r]}{(x_{1,1}-1,x_{2,2}-1,x_{2,1})}.
$$
It follows that $k\l[\G_{a,k}\r]$ can be naturally identified with the algebra $k[x]$
of polynomial in the variable $x$ and coefficients in $k$.
One can define in an analogous way the affine group scheme $\SL_{n,k}$ over $k$, whose algebra is
$k\l[\SL_{n,k}\r]=\frac{k[x_{i,j},\, i,j=1,\dots, n]}{(\det(x_{i,j})-1)}$.
\end{exa}

The Yoneda Lemma ensures that $k\l[G\r]$ is unique up to isomorphism.  An important property (that we won't use, because we are not getting
into the details of the proofs) of $k\l[G\r]$
is that it has a natural structure of Hopf algebra.
\par
The affine group scheme $G$ is said to be an algebraic group
if $k\l[G\r]$ is a finitely generated $k$-algebra.
This means that $k\l[G\r]$ can be identified with a quotient $k[x_1,\dots, x_n]/I $
of a ring of polynomials by a convenient (Hopf) ideal $I$.  It allows to identify
$G(B)$ with the set of zeros of $I$ in $B^n$, for any $k$-algebra $B$.
In other words, $G$ can be identified with an affine variety defined over $k$.
All the affine group schemes in Example \ref{exa:groupschemes}, as well as all the affine group schemes appearing in this paper, are algebraic groups.
\par
If $G$ is an algebraic group and  $G^\p$ is another affine group scheme defined over $k$, we say that $G^\p$ is an affine subgroup scheme of $G$
if there exists a surjective morphism of Hopf algebras $k\l[G\r]\to k\l[G^\p\r]$.
This implies that $G^\p(B)$ can be identified naturally to a subgroup of $G(B)$, for any $k$-algebra $B$.
We say that $G^\p$ is a normal algebraic subgroup of $G$, if $G^\p(B)$ is a normal subgroup of $G(B)$,
for any $k$-algebra $B$.

\begin{exa}
In the notation of Example \ref{exa:groupschemes}, we have:
\begin{trivlist}
\item
1. The additive affine group scheme $\G_{a,k}$ is an algebraic subgroup of $\GL_{2,k}$.
\item
2. We have a surjective morphism from $k\l[\GL_{n,k}\r]$ to $k\l[\SL_{n,k}\r]$
defined by $t\mapsto 1$, therefore $\SL_{n,k}$ is naturally an algebraic subgroup of $\GL_{n,k}$.
\end{trivlist}
\end{exa}

For further reference, we describe the algebraic subgroups of  $\G_{m,k}$.
As we have already pointed out, we have $k[\G_{m,k}]=\frac{k[x,t]}{(xt-1)}$, therefore we can write for short
$k[\G_{m,k}]=k\l[x,\frac{1}{x}\r]$.
The algebraic subgroups of $\G_{m,k}$ are the defined by equations of the form $x^n-1$, for any non-negative integer $n$.
They are represented by the quotients $\frac{k\l[x,\frac{1}{x}\r]}{(x^n-1)}$.
If $G$ is one of those subgroups, with $n\geq 1$, and $B$ is a $k$-algebra,
then $G(B)$ is nothing more that the group of $n$-th roots of unity contained in $B$.
For $n=1$, we obtain the trivial algebraic subgroup $\{1\}$ of $\G_{m,k}$, while for $n=0$ we obtain the whole $\G_{m,k}$.
\par
Let us consider the algebraic group $G_{m,k}^n$, for some positive integer $n$. We have $k[G_{m,k}^n]=k\l[x_1,\frac{1}{x_1},\dots,x_n,\frac{1}{x_n}\r]$.
The algebraic subgroups of $G_{m,k}^n$ are defined by polynomials of the form
$x_1^{\a_1}\dots x_n^{\a_n}-1$, where $\a_1,\dots,\a_n\in\Z$.

\medskip
We will also need the description of the algebraic subgroups of $\G_{a,k}^n$, where $n$ is a positive integer.
For any $k$-algebra $B$, $\G_{a,k}^n(B)$ can be naturally identified to $B^n$.
Therefore, an algebraic subgroup $G$ of $\G_{a,k}^n$ is defined by an ideal generated by at most $n$ independent linear equations with coefficients in $k$.
In particular, we will use the fact that a proper algebraic subgroup of $\G_{a,k}^n$ is always an algebraic subgroup of the group represented by
the algebra $\frac{k[x_1,\dots,x_n]}{(\a_1x_1+\dots+\a_nx_n)}$, for some $\a_1,\dots,\a_n\in k$, not all zero.
Notice that, for $n=1$, the algebraic group $G_{a,k}$ does not have any proper algebraic subgroup.

\subsection{The Galois group of a linear difference system}

Let $K\subset F$ be our base field, with $k=K^\tau=F^\tau$, and let us consider a system of the form \eqref{eq:system}.
From now on we will assume \emph{implicitly} that all \PV\ rings and all \PV\ fields are contained in $F$.
So let $R$ ($\subset F$) be a \PV\ ring for \eqref{eq:system}
and let $L$ be its field of fractions.
\par
For more details
on what follows, see
\cite[\S 2.7]{OvchinnikovWibmerSGaloisTheoryofLinearDifferenceEquations}.\footnote{
In the notation of \cite{OvchinnikovWibmerSGaloisTheoryofLinearDifferenceEquations},
one has to take $\Phi=\tau$ and $\sigma$ to be the identity.}

\begin{defn}
\emph{\cite[2.50]{OvchinnikovWibmerSGaloisTheoryofLinearDifferenceEquations}}
We call the (difference) Galois group of \eqref{eq:system} the following group scheme:
$$\begin{array}{rccc}
\Gal(L/K):&\hbox{\rm $k$-Algebras}&\to&\hbox{\rm Groups}\\
&B&\mapsto&\Aut^\tau(R\otimes_k B/K\otimes_k B),
\end{array}
$$
where:
\begin{enumerate}
\item
the $k$-algebra $B$ is endowed with a structure of trivial $\tau$-$k$-algebra,
so that $\tau(f\otimes b)=\tau(f)\otimes b$,
for any $f\in R$ and any $b\in B$;
\item
$\Aut^\tau(R\otimes_k B/K\otimes_k B)$ is the group of the ring automorphisms of
$R\otimes_k B$, that fix $K\otimes_k B$ and commute with $\tau$.
\end{enumerate}
The functor $\Gal(L/K)$ acts on morphisms by extension of constants, namely,
each morphism of $k$-algebras $\a:B_1\to B_2$ defines a structure
of $B_1$-algebra over $B_2$ and the
definition of
$\Gal(L/K)(\a): \Aut^\tau(R\otimes_k B_1/K\otimes_k B_1)
\to\Aut^\tau(R\otimes_k B_2/K\otimes_k B_2)$ relies on the fact
that $R\otimes_k B_2\cong R\otimes_k B_1\otimes_{B_1,\a} B_2$.
\end{defn}

For any choice of a fundamental solution matrix
$U\in\GL_d(R)$ of \eqref{eq:system}, for any $k$-algebra $B$ and any $\varphi\in\Gal(L/K)(B)$
we have that
$$
\tau(U^{-1}\varphi(U))=U^{-1}A^{-1}\varphi(A)\varphi(U)=U^{-1}\varphi(U)\in\GL_d(B),
$$
where we have identified $U$ and $U\otimes 1$ in $R\otimes_k B$, making an abuse of notation that we will repeat
frequently. (Notice that we have used the fact that $(R\otimes_k B)^\tau=B$. See Lemma \ref{lemma:LinDisj}.)
The maps $\varphi\mapsto U^{-1}\varphi(U)$ represents $\Gal(L/K)(B)$ as a subgroup of $\GL_d(B)$.
The linearity of the difference system \eqref{eq:system} immediately implies that another choice of the
fundamental solution matrix leads to a conjugated representation,
so that  most of the times we can identify $\Gal(L/K)(B)$ with a subgroup scheme of $\GL_{d,k}(B)$,
forgetting to mention the matrix $U$.
The following proposition says that such a representation is functorial in $B$, in the sense
that $\Gal(L/K)$ is an algebraic subgroup of $\GL_{d,k}$, as in the next example.

\begin{exa}\label{exa: order1}
Let us consider the rank-one difference equation $\tau(y)=ay$, where $a\in K$, as in Example \ref{exa: order1-PV}.
By assumption, there exists $z\in F$ such that $\tau(z)=az$ and $R=K[z, z^{-1}]$ is its \PV\ ring.
For any $k$-algebra $B$ and any $\varphi\in\Gal(L/K)(B)$, the element $\varphi(z\otimes 1)$ of $R\otimes_k B$
must be a solution of $\tau(y)=ay$, and hence there exists $c_\varphi\in B^\ast$ such that
$\varphi(z\otimes 1)=c_\varphi(z\otimes 1)$.
This means that $\Gal(L/K)$ can be identified with a subgroup of the multiplicative group $\G_{m,k}$ defined over $k$,
therefore it coincides either with $\G_{m,k}$ or with a cyclic group.
If for instance $a=-1$, then we must have $z^2\in k$, and therefore $c_\varphi^2=1$.
\end{exa}

\begin{prop}\label{prop:algebraicgp}
\emph{\cite[Lemma 2.51]{OvchinnikovWibmerSGaloisTheoryofLinearDifferenceEquations}}
The Galois group $\Gal(L/K)$ is an algebraic group defined over $k$,
represented by the $k$-algebra $(R\otimes_K R)^\tau$.
\end{prop}

\begin{rmk}
We remind that the statement above means that
$(R\otimes_K R)^\tau$ is a finitely generated $k$-algebra and
that the functors $\Gal(L/K)$ and $\hbox{\rm Hom}_{\hbox{\rm $k$-Algebra}}((R\otimes_K R)^\tau, -)$
are naturally isomorphic. 
In particular for any
$k$-algebra $B$, we have can identify $\Gal(L/K)(B)$ with $\hbox{\rm Hom}_{\hbox{\rm $k$-Algebra}}((R\otimes_K R)^\tau, B)$.
\end{rmk}

The proposition above says that there exists an ideal $I$ of the ring of polynomials
$k[X,\det X^{-1}]$, with $X=(x_{i,j})_{i,j=1,\dots,d}$, such that
for any $k$-algebra $B$ the image of the group morphism defined above
$$
\begin{array}{ccc}
  \Gal(L/K)(B)& \to & \GL_d(B) \\
  \varphi & \mapsto & [\varphi]_U:=U^{-1}\tau(U)
\end{array}
$$
is exactly the set of zeros of $I$ in $\GL_d(B)$.
The idea of the proof is to consider the $k$-algebra $k[Z,\det Z^{-1}]\hookrightarrow (R\otimes_kR)^\tau $, where
$Z:=(U^{-1}\otimes 1)(1\otimes U)$. Then one can prove that
we have the series of isomorphisms:
$$
R\otimes_K R\cong R. k[Z,\det Z^{-1}]\cong R\otimes_k\l(R\otimes_K R\r)^\tau.
$$
This allows to prove a series of group isomorphisms showing that for any $k$-algebra $B$ we have:
$$
\Hom_{\hbox{\rm $\tau$-$(K\otimes_k B)$-alg}}(R\otimes_k B, R\otimes_k B)\cong \Hom_{\hbox{\rm $k$-alg}}((R\otimes_K R)^{\tau},B).
$$
See \cite{OvchinnikovWibmerSGaloisTheoryofLinearDifferenceEquations} for details.

\begin{exa}\label{exa:Theta}
Let $F$ be the field of meromorphic functions over $\C^\ast$ and let $\tau$ be defined by $\tau: f(x)\mapsto f(qx)$,
for $q\in \C$, such that $|q|>1$. The field of constants $k$ is the field of meromorphic functions over the torus $\C^\ast/q^\Z$ and we set $K=k(x)$.
The Jacobi Theta function
$$
\Theta(x)=\sum_{x\in\Z}q^{-n(n+1)/2}x^n\in F
$$
is solution of the difference equation $\tau(y)=xy$. Its differential Galois group $G$ is the multiplicative group. Indeed
for any $k$-algebra $B$ and for any $\varphi\in G(B)$, $\varphi$ multiplies $\Theta(x)$ by an invertible element of $B$.
On the other hand, since $\Theta(x)$ is a transcendental function, any invertible constant of $B$ defines an automorphism
of $K[\Theta(x),\Theta(x)^{-1}]\otimes_kB$.
\par
Now let us consider an integer $r\geq 2$  and choose a $r$-th root $q^{1/r}$ of $q$.
The meromorphic function $z(x):=\Theta(q^{1/r}x)/\Theta(x)\in F$
is a solution of the finite difference equation $\tau(y)=q^{1/r}y$ and its difference Galois group
is the cyclic subgroup of $G_{m,k}$ of order $r$.
To prove the last claim it is enough to notice that $z(qx)^r =q z(x)^r$, hence $z(x)^r$ is a meromorphic function
of the form $xg(x)$, with $g(x)\in k\subset K$.
 \end{exa}

\begin{exa}\label{exa:additivegroup}
Let us consider an element $f\in K$ and the inhomogeneous difference equation $\tau(y)=y+f$.
By assumption, there exists a solution $z\in F$ and we have already noticed that $R=k[z]$.
See Example \ref{exa:log}.
For any $k$-algebra $B$ and any $\varphi\in\Gal(L/K)$, the element $\varphi(z)$ of $R\otimes_k B$
must be another solution of $\tau(y)=y+f$, hence there exists $c_\varphi\in B$ such that
$\varphi(z)=z+c_\varphi$ (here we have identified $z$ and $z\otimes 1$).
It follows that $\Gal(L/K)$ is a algebraic subgroup of the additive group $\G_{a,k}$.
This means that either $\Gal(L/K)=\{1\}$ or $\Gal(L/K)=\G_{a,k}$.
\end{exa}

\subsection{Transcendence degree of the \PV\ extension}

We can now state the first important result from the point of view of applications to number theory and
more specifically to transcendence.
It compares the dimension of $G$ over $k$ as an algebraic variety and the transcendence degree of $R$ over $K$
and its proof is based on Proposition \ref{prop:algebraicgp}.

\begin{thm}\emph{\cite[Lemma 2.53]{OvchinnikovWibmerSGaloisTheoryofLinearDifferenceEquations}}
\label{thm:dimension}
Let $R$, $L$ and $G$ be as above. Then the dimension of $G$ as an algebraic variety over $k$ is equal to the transcendence
degree of $R$ (or of $L$) over $K$:
$$
\dim_k G=\hbox{\rm trdeg}_K R=\hbox{\rm trdeg}_K L.
$$
\end{thm}

\begin{exa}\label{exa:Gamma}
Let us consider the case of finite difference equations, i.e, let $F$ be the field of meromorphic functions over $\C$
equipped with the operator $\tau : f(x)\mapsto f(x+1)$. Then $k$ is the field of $1$-periodic functions and we set $K=k(x)$.
We consider the difference equations $\tau(y)=x y$, which is satisfied by the Euler Gamma function $\Gamma(x)$.
Its Picard-Vessiot ring is $K[\Gamma(x),\Gamma(x)^{-1}]$, as discussed in Example \ref{exa: order1-PV}.
As in Example \ref{exa: order1}, its Galois group $G$ is a subgroup of $\G_{m,k}$.
For any $k$-algebra $B$ and any $\varphi\in G(B)$, there exists an invertible element $c_\varphi$ of $B$ such that $\varphi(\Gamma)=c_\varphi\Gamma$.
As already explained, the proper subgroups of  $\G_{m,k}$ are the finite cyclic groups. If $G$ was a the cyclic group,
$\Gamma(x)$ would be an algebraic function, by the previous theorem. Proving that the functional equations of the Gamma function implies that
it is not algebraic over $k(x)$ is an exercise that we leave to the reader.
Therefore
the difference Galois group $G$ is the whole $\G_{m,k}$.
\end{exa}

\begin{exa}
Let us consider the system
$$
\tau(Y)=\begin{pmatrix}
          x & 1 & 0 \\
          0 & x & 1 \\
          0 & 0 & x
        \end{pmatrix}Y.
$$
We denote by $(\cdot)^\p$ the derivation $\frac d{dx}$, with respect to $x$.
Since $\tau$ and $\frac{d}{dx}$ commute, we have:
$$
\tau\bigg(\Gamma^\p(x)\bigg)=\frac{d}{dx}\bigg(x\Gamma(x)\bigg)=x\Gamma^\p(x)+\Gamma(x)
$$
and
$$
\tau\l(\frac{\Gamma^{\p\p}(x)}{2}\r)=\frac{1}{2}\frac{d}{dx}\bigg(x\Gamma^\p(x)+\Gamma(x)\bigg)=x\frac{\Gamma^{\p\p}(x)}{2}+\Gamma^\p(x).
$$
Therefore a solution matrix is given by
$$
Y=\begin{pmatrix}
  \Gamma(x) & \Gamma^\p(x) &\Gamma^{\p\p}(x)/2 \\
  0 & \Gamma(x) & \Gamma^\p(x) \\
  0 & 0 & \Gamma(x)
\end{pmatrix},
$$
so that the associated \PV\ ring is
$R= K\big[\Gamma(x), \Gamma^\p(x) ,\Gamma^{\p\p}(x), \Gamma(x)^{-1}\big]$.
Let $G$ be its difference Galois group and let $B$ be a $k$-algebra. For any $\varphi\in G(B)$, $\varphi$ commutes to the action of $\tau$
over $R\otimes_kB$, therefore it must send any element of $R$, which is solution of a $\tau$-difference equation, into a solution of the same  equation.
We know that $\Gamma$ is solution of a homogenous order $1$ equation,
while $\frac{\Gamma^\p(x)}{\Gamma(x)}$ and $\frac{d}{dx}\l(\frac{\Gamma^\p(x)}{\Gamma(x)}\r)$
are solutions inhomogeneous order $1$ equations.
Therefore, as in Examples \ref{exa: order1} and \ref{exa:additivegroup},
there must exist $c_\varphi\in\G_{m,k}(B)$ and $(d_{1,\varphi},d_{2,\varphi})\in \G_{a,k}(B)^2$ such that
$$
\l\{\begin{array}{l}
  \varphi\big(\Gamma(x)\big)=c_\varphi\Gamma(x),\\~\\
  \ds \varphi\l(\frac{\Gamma^\p(x)}{\Gamma(x)}\r)=\frac{\Gamma^\p(x)}{\Gamma(x)}+d_{1,\varphi},\\~\\
  \ds \varphi\l(\frac{d}{dx}\l(\frac{\Gamma^\p(x)}{\Gamma(x)}\r)\r)=\frac{d}{dx}\l(\frac{\Gamma^\p(x)}{\Gamma(x)}\r)+d_{2,\varphi}
    =\frac{\Gamma^{\p\p}(x)}{\Gamma(x)}-\frac{\Gamma^\p(x)}{\Gamma^2(x)}+d_{2,\varphi}.
\end{array}\r.
$$
So we have represented $G$ as a subgroup of $\G_{m,k}\times \G_{a,k}^2$.
Since $\Gamma(x), \Gamma^\p(x), \Gamma^{\p\p}(x)$ are algebraically independent by Hölder theorem \cite{Holder1887}
(see Proposition \ref{prop:Holder} below for a proof), it is actually an isomorphism, thanks to Theorem \ref{thm:dimension}.
The construction above can be easily generalized to system associated to a Jordan block of eigenvalue $x$ and order higher than 3,
using higher order derivatives of $\Gamma$.
\par
Now let us consider the
finite difference equation $(\tau-x)^n(y)=0$, for some positive integer $n$. Such an equation occurs in
generalized Carlitz modules studied in \cite{Hellegouarch-Recher-gen-t-modules} and more extensively in \cite{PellarinGenCarlitz}.
We have seen that $\Gamma$ is a solution for $n=1$. One can verify recursively that,
for $n>1$, a basis of solution over $k$ is given by the higher order derivatives of $\Gamma$ with respect to $x$,
namely the set
$\Gamma(x),\Gamma^\p(x),\dots,\Gamma^{(n-1)}(x)$.
It follows that the \PV\ ring is the same as the one of the system above, namely
$R=\big[\Gamma(x),\Gamma^\p(x) ,\dots,\Gamma^{(n-1)}(x), \Gamma(x)^{-1}\big]$.
We conclude that the
Galois group is $\G_{m,k}\times \G_{a,k}^{n-1}$.
\end{exa}

Before being able to prove the most common results used in the applications, we need to state the Galois correspondence and its properties.

\section{The Galois correspondence (first part)}

We consider $R,L\subset F$ and $G:=\Gal(L/K)$ as above.

\begin{defn}
Let $\frac r s\in L$, with $r,s\in R$ and $s\neq 0$, and let $\varphi\in G(B)$, for a $k$-algebra $B$.
We say that $\frac r s$ is invariant under the action of $\varphi$ if in $R\otimes_k B$ we have:
$$
\varphi(r\otimes 1)(s\otimes 1)=(r\otimes 1)\varphi(s\otimes 1).
$$
If $H$ is an algebraic subgroup of $G$ defined over $k$, then $\frac r s$ is invariant under the action of $H$ if
$\frac r s$ is invariant under the action of $\varphi$, for all $k$-algebras $B$ and all $\varphi\in H(B)$.
\par
We denote by $L^H$ the set of elements of $L$ invariant under the action of $H$.
\end{defn}

\begin{rmk}
Notice that if $\varphi\in G(k)$, then the condition above simply means that $\frac{\varphi(r)}{\varphi(s)}=
\frac{r}{s}$ in $L$.
\end{rmk}

If $M$ is an intermediate field of $L/K$, stable by $\tau$, then we can consider
\eqref{eq:system} as a system defined over $M$.
Indeed if $L$ is a \PV\ field over $K$, it must be a \PV\ field over $M$
and we can define $\Gal(L/M)$.

\begin{thm}\label{thm:GaloisCorrespondence1}
\emph{\cite[2.52]{OvchinnikovWibmerSGaloisTheoryofLinearDifferenceEquations}}
There exists a one-to-one correspondence between the algebraic subgroup of $G$ defined over $k$ and the
intermediate fields of $L/K$ stable by $\tau$. In the notation above we have two maps that are one the inverse of the other and are defined by:
$$
H\mapsto L^H,\ M\mapsto\Gal(L/M).
$$
Moreover if $H$ is an algebraic subgroup of $G$, then $H=G$ if and only if $L^H=K$.
\end{thm}

As usual in Galois theories, the last statement is the key-point of the Galois correspondence.

\section{Application to transcendence and differential transcendence}

\subsection{General statements}

From the theorems above we deduce a first result on transcendency that is very useful in many settings.
The criteria of transcendence below are contained in \cite[\S3]{CharlotteSinger:deltadeltapi}, up to some reformulations.
They are the key-point of several applications, for instance in
\cite{dreyfus-hardouin-rosuqes-jems},
\cite{dreyfus-hardouin-roques},
\cite{Dreyfus-Hardouin-Roques-Singer-genus1},
\cite{Dreyfus-Hardouin-Roques-Singer-genus0},
\cite{Dreyfus-Hardouin-t}.
\par
In the proposition below the assumption $k=F^\tau=K^\tau$ is crucial as well as in all the subsequent results in this section.

\begin{prop}\label{prop:transcendenceINH}
Let $f_1,\dots, f_d\in K^\ast$ and let $z_1,\dots,z_d\in F$ be a solution of the
following inhomogeneous difference
system:
\beq\label{eq:transcendenceINH}
\Big\{\tau(z_i)=z_i+f_i,\hbox{~for $i=1,\dots,d$.}
\eeq
The following assertions are equivalent:
\begin{enumerate}
  \item There exist $\la_1,\dots,\la_d\in k$, not all zero, and $g\in K$ such that $\la_1f_1+\dots+\la_d f_d=\tau(g)-g$.
  \item There exist $\la_1,\dots,\la_d\in k$, not all zero, such that $\la_1z_1+\dots+\la_d z_d\in K$.
  \item There exist $\la_1,\dots,\la_d\in K$, not all zero, such that $\la_1z_1+\dots+\la_d z_d\in K$.
  \item $z_1,\dots,z_d$ are algebraically dependent over $K$.
\end{enumerate}
\end{prop}

\begin{rmk}
Notice that the first statement above is about the $f_i$'s, while the others are about the $z_i$'s.
\end{rmk}

\begin{proof}
Let us assume that we are in the situation of the first assertion.
We have:
$$
\tau\l(\sum_{i=1}^{d}\la_iz_i-g\r)=\sum_{i=1}^{d}\la_iz_i-\tau(g)+\sum_{i=1}^{d}\la_if_i=\sum_{i=1}^{d}\la_iz_i-g,
$$
hence $\sum_{i=1}^{d}\la_iz_i-g\in k$, which proves 2.
Moreover, the implications $2\Rightarrow 3\Rightarrow 4$ are tautological.
\par
We conclude by proving that $4\Rightarrow 1$.
As in Example \ref{exa:additivegroup}, the system \eqref{eq:transcendenceINH} is equivalent to the following linear system
of order $2d$:
$$
\tau(Y)=\begin{pmatrix}
          \fbox{$\begin{matrix}1 & f_1 \\0 & 1\end{matrix}$} & 0 & \cdots & 0 \\
           0 & \fbox{$\begin{matrix}1 & f_2 \\0 & 1\end{matrix}$} &\ddots  &\vdots  \\
           \vdots & \ddots & \ddots & 0 \\
           0 & \cdots & 0 & \fbox{$\begin{matrix}1 & f_d \\0 & 1\end{matrix}$}
        \end{pmatrix}Y,
$$
so that its \PV\ ring is $R=k[z_1,\dots,z_d]$. It follows that $\Gal(L/K)$ is an algebraic subgroup of $\G_{a,k}^d$, defined over $k$, and that
for any $k$-algebra $B$ and any
$\varphi\in\Gal(L/K)(B)$ there exists $c_{\varphi,1},\dots,c_{\varphi,d}\in B$ such that
$\varphi(z_i)=z_i+c_{\varphi,i}$.
\par
Since $z_1,\dots,z_d$ are algebraically dependent over $K$, Theorem \ref{thm:dimension} implies that the dimension of $\Gal(L/K)$
is strictly smaller than $d$ and  hence $\Gal(L/K)$ is a proper subgroup scheme of $\G_{a,k}^d$.
All the proper algebraic subgroup of $\G_{a, k}^d$ are contained in a hyperplane and $\Gal(L/K)$ is defined over $k$,
therefore there exist $\la_1,\dots,\la_d\in k$, not all zero, such that for any $k$-algebra $B$ and any
$\varphi\in\Gal(L/K)(B)$, we have $\sum_{i=1}^{d}\la_i c_{\varphi,i}=0$.
We conclude that $g:=\sum_{i=1}^{d}\la_iz_i$ verifies
$$
\varphi(g)=\sum_{i=1}^{d}\la_iz_i+\sum_{i=1}^{d}\la_ic_{\varphi,i}=\sum_{i=1}^{d}\la_iz_i=g,
$$
and hence that $g\in K$, by the Galois correspondence.
Finally we have:
$$
\tau(g)-g=\tau\l(\sum_{i=1}^{d}\la_iz_i\r)-\sum_{i=1}^{d}\la_iz_i=\sum_{i=1}^{d}\la_if_i.
\$$endproof

In an analogous way, taking into account that the algebraic subgroup of $\G_{m,k}^d$
are defined by equations of the form $x_1^{\a_1}\cdots x_d^{\a_d}=1$,
for some $\a_1,\dots\a_d\in\Z$, it is possible to prove the following proposition:

\begin{prop}\label{prop:AlgDepHomo}
Let $a_1,\dots a_d\in K^\ast$ and let $z_1,\dots,z_d\in F^\ast$ be a solution of the
following difference
system:
\beq\label{eq:transcendenceH}
\Big\{\tau(z_i)=a_iz_i,\hbox{~for $i=1,\dots,d$.}
\eeq
The following assertions are equivalent:
\begin{enumerate}
  \item There exist $\la_1,\dots,\la_d\in \Z$, not all zero, and $g\in K$ such that $a_1^{\la_1}\cdots a_d^{\la_d }=\tau(g)/g$.
  \item There exist $\la_1,\dots,\la_d\in \Z$, not all zero, such that $z_1^{\la_1}\cdots z_d^{\la_d }\in K$.
  \item $z_1,\dots,z_d$ are algebraically dependent over $K$.
\end{enumerate}
\end{prop}

\begin{rmk}
Notice that Proposition \ref{prop:AlgDepHomo} has one more characterization of the algebraic dependency of the solutions.
Here we only have $3$ assertions because of the multiplicative form of the algebraic subgroups of $\G_{m,k}^d$.
\end{rmk}

\subsection{Differential algebraicity and $D$-finiteness}
\label{subsec:Dfinite}

We now switch our attention to the characterization of  differential algebraicity, hence
in this subsection we assume that we are in characteristic zero.
We assume that the field $F$ comes equipped
with a derivation $\partial$ that commutes with the endomorphism $\tau$. This implies in particular that
$\partial$ induces a derivation on both $k$ and $K$.

\begin{exa}
In the notation of Example \ref{exa:basic}, we can take $\partial=\frac d {dx}$ for $\tau: f(x)\mapsto f(x+1)$ and
and $\partial=x\frac d{dx}$ for $\tau:f(x)\mapsto f(qx)$.
\end{exa}

We recall the following definition:

\begin{defn}
We say that $f\in F$ is differentially algebraic (with respect to $\partial$) over $K$, if
there exists an integer $n\geq 0$ such that $f,\partial(f),\dots,\partial^n(f)$ are algebraically dependent over $K$, or,
equivalently,  if $f$ is solution of an algebraic differential equation over $K$.
We say that $f$ is differentially transcendental over $K$ if it is not differentially algebraic over $K$ and that
it is $D$-finite over $K$ if it is differentially algebraic and, moreover, it is the solution  of a linear differential equation
with coefficients in $K$.
\par
We say that $F$ is differentially algebraic over $K$ is all elements of $F$ are differentially algebraic over
$k$.
\end{defn}

\noindent
Once again, the assumption $F^\tau=K^\tau$ is crucial in the following corollaries:

\begin{cor}\label{cor:DiffAlgGa}
Let $f\in K^\ast$ and $z\in F$ be  a solution of $\tau(y)=y+f$. The following statements are equivalent:
\begin{enumerate}
  \item There exist $n\geq 0$, $\la_0,\dots,\la_n\in k$, not all zero, and $g\in K$ such that $\la_0f+\la_1\partial(f)+\dots+\la_n \partial^n(f)=\tau(g)-g$.
  \item There exist $n\geq 0$, $\la_0,\dots,\la_n\in k$, not all zero, such that $\la_0z+\la_1\partial(z)+\dots+\la_n \partial^n(z)\in K$.
  \item $z$ is $D$-finite over $K$.
  \item $z$ is differentially algebraic over $K$.
\end{enumerate}
\end{cor}

\noindent
In particular, the corollary above says that:

\begin{cor}
In the notation of Corollary \ref{cor:DiffAlgGa},
if $z$ is not $D$-finite over $K$ then $z$ is differentially transcendental over $K$.
\end{cor}

\begin{proof}[Proof of Corollary \ref{cor:DiffAlgGa}.]
The assumption on the commutativity of $\partial$ and $\tau$ implies that $z$ satisfies all the following difference equations:
$$
\tau\l(\partial^i(z)\r)=\partial^i(z)+\partial^i(f),
\hbox{~for all $i=0,1,2,\dots$. }
$$
The statement follows from Proposition \ref{prop:transcendenceINH}.
\end{proof}

\begin{cor}\label{cor:DiffAlgGm}
Let $a\in K^\ast$ and $z\in F^\ast$ be  a solution of $\tau(y)=ay$. The following statement are equivalent:
\begin{enumerate}
  \item There exist $n\geq 0$, $\la_0,\dots,\la_n\in k$, not all zero, and $g\in K$ such that $\la_0\frac{\partial(a)}{a}+\la_1\partial\l(\frac{\partial(a)}{a}\r)
    +\dots+\la_n \partial^n\l(\frac{\partial(a)}{a}\r)=\tau(g)-g$.
  \item There exist $n\geq 0$, $\la_0,\dots,\la_n\in k$, not all zero, such that $\la_0\frac{\partial(z)}{z}+\la_1\partial\l(\frac{\partial(z)}{z}\r)
    +\dots+\la_n \partial^n\l(\frac{\partial(z)}{z}\r)\in K$.
  \item $\frac{\partial(z)}{z}$ is $D$-finite over $K$.
  \item $z$ is differentially algebraic over $K$.
\end{enumerate}
\end{cor}

\begin{proof}
Notice that $z$ is differentially algebraic over $K$ if and only if $\partial(z)/z$ is differentially algebraic
over $K$.
In fact, $\partial^n\l(\frac{\partial(z)}{z}\r)\in\frac{1}{z^n}K[z,\partial(z),\dots,\partial^n(x)]$, therefore an elementary algebraic manipulation allows to
transform an algebraic differential equation satisfied by $z$ into an algebraic differential equation satisfied by $\partial(z)/z$ and \emph{vice versa}.
Taking the logarithmic derivative of $\tau(z)=az$, we obtain:
$$
\tau\l(\frac{\partial(z)}{z}\r)=\frac{\partial(z)}{z}+\frac{\partial(a)}{a}.
$$
The statement follows from the Corollary \ref{cor:DiffAlgGa}.
\end{proof}

\begin{rmk}
The generalization of the last two corollaries to systems of order 1 equations and to an arbitrary set of commuting derivations is straightforward.
For the generalization to the case of equation of the form $\tau(y)=ay+f$, we refer to \cite[Propositions 3.8, 3.9, and 3.10]{CharlotteSinger:deltadeltapi}.
\end{rmk}

\section{Applications to special cases}
\label{sec:applications}

In this section we suppose that $F$ has characteristic zero.
The results in \S\ref{subsec:holder} and \S\ref{subsec:ishizaki} have been originally proven in \cite[\S3]{CharlotteSinger:deltadeltapi}.

\subsection{Finite difference equations and Hölder theorem}
\label{subsec:holder}

We are in the situation of Example \ref{exa:Gamma}, i.e., let $F$ be the field of meromorphic functions over $\C$
equipped with the operator $\tau : f(x)\mapsto f(x+1)$. Then $k$ is the field of meromorphic $1$-periodic functions and we set $K=k(x)$.
We set $\partial=\frac d{dx}$, which commutes with $\tau$.

\begin{cor}\label{cor:FiniteDifferenceGa}
In the notation above, let $f\in K^\ast$ and $z\in F$ be such that $\tau(z)=z+f$.
The following assertions are equivalent:
\begin{enumerate}
  \item $z$ is differentially algebraic over $K$.
  \item $z$ is $D$-finite over $K$.
  \item There exist a positive integer $n$, $\la_0,\dots,\la_n\in k$ and $g\in K$ such that
  $$
  \la_0 f+\la_1\partial\l(f\r)+\dots+\la_n\partial^n\l(f\r)=g(x+1)-g(x).
  $$
\end{enumerate}
If $f\in\C(x)$ (resp. $f\in\overline{\Q}(x)$), then they are also equivalent to:
\begin{enumerate}
  \item[4.] There exist a positive integer $n$, $\la_0,\dots,\la_n\in\C$ (resp. $\in\overline{\Q}$) and $g\in \C(x)$
  (resp. $\in\overline{\Q}(x)$) such that
  $$
  \la_0 f+\la_1\partial\l(f\r)+\dots+\la_n\partial^n\l(f\r)=g(x+1)-g(x).
  $$
\end{enumerate}
\end{cor}

\begin{proof}
Notice that $1\Leftrightarrow 2\Leftrightarrow 3$  follow from Corollary \ref{cor:DiffAlgGa}.
Moreover $4\Rightarrow 3$ is trivial. Let us prove that $3\Rightarrow4$, by a classical descent argument.
Let $C=\C$ or $\overline{\Q}$, so that $f\in C(x)$. Let $N$ be the degree of the denominator of $g$ and $M$ the degree of its denominator.
We consider a ring of polynomials of the form $C[\La_0,\dots,\La_n,A_{0},\dots,A_N,B_{0}\dots,B_M]$,
so that we can write the equality
\beq\label{eq: Holder}
\La_0f+\La_1\partial\l(f\r)+\dots+\La_n\partial^n\l(f\r)=\frac{A_0+A_1(x+1)+\dots+A_N(x+1)^N}{B_0+B_1(x+1)+\dots+B_M(x+1)^M}-\frac{A_0+A_1x+\dots+A_Nx^N}{B_0+B_1x+\dots+B_Mx^M}.
\eeq
Equalizing the coefficients of each integer powers of $x$ in \eqref{eq: Holder}, we obtain a system of polynomial equations with coefficients in $C$,
that has a solution in $k$, by assumption. Since $C$ is an algebraically closed field
contained in $k$,
it must also have a solution in $C$.   This proves the corollary. 
\end{proof}

Although the last assertion of Corollary \ref{cor:FiniteDifferenceGa} is stated over $\C$ (or over $\overline{\Q}$), \
we cannot conclude the differential algebraicity of $z$ over $\C(x)$.
See Example \ref{exa:nondescentoverC} that it is based on the fact that there are meromorphic $1$-periodic
functions that are differentially transcendental  over $\C(x)$.
For now, notice that the statement above only implies the following:

\begin{cor}\label{cor: GeneralHolder}
We consider the same notation as in the previous corollary, with $C=\overline{\Q}$ or $\C$ and $f\in C(x)$.
Suppose that for any $n\geq 0$, any $\la_0,\dots,\la_n\in C$ and any $g\in C(x)$, we have: $\la_0 f+\la_1\partial\l(f\r)+\dots+\la_n\partial^n\l(f\r)\neq g(x+1)-g(x)$.
Then $z$ is differentially transcendent over $K$ and hence over $C(x)$.
\end{cor}

The Euler Gamma function, that we have already mentioned in some examples, is a meromorphic function over $\CC$ satisfying
the functional equation $\Ga(x+1)=x\Ga(x)$.
Hölder theorem \cite{Holder1887} says that the Gamma function is differentially transcendental over $\C(x)$ and
we are now able to prove it, using a Galoisian argument that has first appeared in \cite{hardouin_compositio} and \cite{CharlotteSinger:deltadeltapi}.
Notice that in \cite{BankKaufamannGamma2} there is a similar proof of the differential transcendency of the Gamma function,
which relies on a statement similar to Corollary \ref{cor: GeneralHolder} in the specific case of the Gamma function,
proven by an elementary argument of
complex analysis.

\begin{prop}\label{prop:Holder}
The Gamma function $\Gamma$ is differentially transcendental over $\C(x)$.
\end{prop}

\begin{proof}
As in the proof of Proposition \ref{cor:DiffAlgGm}, the Gamma function $\Gamma$ is differentially transcendental over $\C(x)$
if and only if
the function $\psi(x):=\partial (\Gamma)(x)/\Gamma(x)$, that verifies the functional equation
$$
\tau(\psi(x))=\psi(x)+\frac 1x,
$$
is differentially transcendental over $\C(x)$.
Suppose that there exist a positive integer $n$, $\la_0,\dots,\la_n\in\C$ and $g\in \C(x)$ such that
$$
\la_0\frac 1x+\la_1\partial\l(\frac 1x\r)+\dots+\la_n\partial^n\l(\frac 1x\r)=g(x+1)-g(x).
$$
Since the left-hand side has all its poles at $0$, while the right-hand side must have at least a non-zero pole,
we find a contradiction, by Corollary \ref{cor: GeneralHolder}.
\end{proof}

The following is a counterexample, based on Hölder theorem, for the fact that we cannot conclude the differential algebraicity over $\C(x)$ in
Corollary \ref{cor:FiniteDifferenceGa}.

\begin{exa}\label{exa:nondescentoverC}
The meromorphic function $\Gamma(\exp(2i\pi x))$ is $1$-periodic, hence belongs to $k\subset K$, but is not differentially algebraic over $\C(x)$, since it is
the composition of a differentially algebraic function and a differentially transcendental function.
In other words, $K$ itself is differentially transcendental over $\C(x)$.
\end{exa}

\begin{cor}\label{cor:qDiffHomogenous}\emph{\cite[Corollary 3.4]{CharlotteSinger:deltadeltapi}}
Let $a(x)\in\C(x)^\ast$ and let $z$ be a meromorphic function over $\C$ solution of $z(x+1)=a(x)z(x)$.
Then $z(x)$ is differentially algebraic over $k(x)$ if and only if
$a(x)=c\frac{g(x+1)}{g(x)}$, for some $g(x)\in\C(x)$ and $c\in\C$.
\end{cor}

\begin{proof}
First of all, replacing $z(x)$ with $z(x)g(x)^{-1}$, for a convenient $g(x)\in\C(x)$, and $a(x)$ with $a(x)\frac{g(x)}{g(x+1)}$, we
 can suppose that two distinguished poles of $a(x)$ do not differ by an integer.
\par
It follows from Corollary \ref{cor:FiniteDifferenceGa}, that $z(x)$ is differentially algebraic over $k(x)$ if and only if
there exist a positive integer $n$, $\la_0,\dots,\la_n\in\C$ and $g\in \C(x)$ such that
$$
\la_0 \frac{\partial(a)}{a}+\la_1\partial\l(\frac{\partial(a)}{a}\r)+\dots+\la_n\partial^n\l(\frac{\partial(a)}{a}\r)=g(x+1)-g(x).
$$
In the differential relation above, the right hand side must have at least two pole in any $\tau$-orbit where it has a pole.
while the left hand side has at worst one pole per $\tau$-orbit.
We conclude that
$a(x)$ is constant.
\par
On the other hand, if $a(x)=c\frac{g(x+1)}{g(x)}$ and we choose a logarithm $\log c$ of $c$,
a general solution of $y(x+1)=a(x)y(x)$ has the form $z(x)=p(x)\exp(x\log c)$, with $p(x)\in k$. The latter is differentially algebraic over $k(x)$.
\end{proof}

\subsection{Linear inhomogeneous $q$-difference equations of the first order}
\label{subsec:q-difference}

We consider the setting of $q$-difference equations, i.e., $F$ is the field of meromorphic functions over $\C^\ast$,
$q$ is a fixed complex number such that $|q|>1$, $\tau: f(x)\mapsto f(qx)$, $K=k(x)$, with $k=F^\tau$.
We consider the derivation $\partial=x\frac d {dx}$, that commutes with $\tau$.
\par
With respect to differential algebraicity, the case of $q$-difference equations deeply differs from the case of finite difference equation
because of the following property (see Example \ref{exa:nondescentoverC}):

\begin{lemma}\label{lemma:EllipticFctsDiffAlgebraic}
The field of elliptic functions $k$ is differentially algebraic over $\C$.
\end{lemma}

To prove Lemma \ref{lemma:EllipticFctsDiffAlgebraic}, it suffices to write the torus $\C^\ast/q^\Z$ in the form $\C/\Z+i\uptau\Z$, where $q=\exp(2i\pi\uptau)$,
using the exponential function, and
remember that the Weierstrass function $\wp$ is differentially algebraic over $\C(x)$,
which is itself differentially algebraic over $\C$.
\par
For further reference, we state the following corollary which is a consequence of the fact that, if we have a tower of differentially algebraic extensions
$\wtilde k/k^\p$ and $k^\p/k$, then $\wtilde k/k$ is also differentially algebraic:

\begin{cor}\label{cor:DiffALgebraicExtensions}
For a meromorphic function $f\in F$, it is equivalent to be differentially algebraic over
the following fields: $k(x)$, $k$, $\C(x)$, $\C$.
\end{cor}

Taking into account the previous lemma, the proof of the corollary below follows word by word the proof of Corollary \ref{cor:FiniteDifferenceGa}:

\begin{cor}\label{cor:qDifferenceGa}
In the notation above, let $f\in K^\ast$ and $z\in F$ be such that $\tau(z)=z+f$.
The following assertions are equivalent:
\begin{enumerate}
  \item $f$ is differentially algebraic over $K$.
  \item $f$ is differentially algebraic over $\C(x)$.
  \item $f$ is $D$-finite over $K$.
  \item There exist a positive integer $n$, $\la_0,\dots,\la_n\in k$ and $g\in K$ such that
  $$
  \la_0 f+\la_1\partial\l(f\r)+\dots+\la_n\partial^n\l(f\r)=g(qx)-g(x).
  $$
\end{enumerate}
Moreover, if $f\in\overline{\Q}(x)$ (resp. $\C(x)$), they are also equivalent to:
\begin{enumerate}
  \item There exist a positive integer $n$, $\la_0,\dots,\la_n\in\overline{\Q}$ (resp. $\C$) and $g\in\overline{\Q}(x)$ (resp. $\C(x)$) such that
  $$
  \la_0 f+\la_1\partial\l(f\r)+\dots+\la_n\partial^n\l(f\r)=g(qx)-g(x).
  $$
\end{enumerate}
\end{cor}

We finally conclude by proving a result for homogenous order 1 $q$-difference equations:

\begin{cor}\label{cor:qDiffHomogenousBis}
\emph{\cite[Corollary 3.4]{CharlotteSinger:deltadeltapi}}
Let $a(x)\in\C(x)^\ast$ and let $z$ be a meromorphic function over $\C^\ast$ (resp. $\C$) be a solution of
$z(qx)=a(x)z(x)$.
The $z(x)$ is differentially algebraic over $\C(x)$ (or equivalently over $k(x)$) if and only if
$a(x)=cx^n\frac{g(qx)}{g(x)}$, for some $g(x)\in\C(x)$, $n\in\Z$ and $c\in\C$
(resp. $n=0$ and $c\in q^\Z$).
\end{cor}

\begin{proof}
If $a(x)=cx^n\frac{g(qx)}{g(x)}$, then a meromorphic solution in $F$ is given by $z(x)=p(x)\frac{\Theta(cx)}{\Theta(x)}\Theta(x)^ng(x)$,
where $p(x)\in k$ and $\Theta(x)=\sum_{n\in\Z}q^{-n(n+1)/2}x^n\in F$ is the Jacobi Theta function,
which verifies the functional equation $\Theta(qx)=x\Theta(x)$.
Notice that $\partial\l(\frac{\partial(\Theta(x))}{\Theta(x)}\r)\in k$, therefore $z$ is differentially algebraic over $\C(x)$.
In particular, if $n=0$, and $c$ is an integer power of $q$, the solution is also meromorphic at zero.
\par
Let us prove the inverse. We assume that $z$ is meromorphic over $\C^*$.
First of all, replacing $z(x)$ with $z(x)g(x)^{-1}$, for a convenient $g(x)\in\C(x)$, and $a(x)$ with $a(x)\frac{g(x)}{g(qx)}$, we
 can suppose that two distinguished poles of $a(x)$ do not differ by an integer power of $q$.
\par
It follows from Corollary \ref{cor:qDifferenceGa}, that $z(x)$ is differentially algebraic over $\C(x)$ if and only if
there exist a positive integer $n$, $\la_0,\dots,\la_n\in\C$ and $g\in \C(x)$ such that
$$
\la_0 \frac{\partial(a)}{a}+\la_1\partial\l(\frac{\partial(a)}{a}\r)+\dots+\la_n\partial^n\l(\frac{\partial(a)}{a}\r)=g(qx)-g(x).
$$
The differential relation above shows that $a(x)$ must have at least two poles in any non-zero $\tau$-orbit, which is in
contradiction with our assumptions, therefore we conclude that
$a(x)=cx^n$, for some $c\in\C$ and $n\in\Z$.
\par
If moreover $z$ is has a pole at zero, rather than an essential singularity, we can take the expansion of $z$ in $\C((x))$.
Plugging it into the equation $z(qx)=cx^nz(x)$,
we see that $n=0$ and hence that $c$ must be an integer power of $q$.
\end{proof}

\subsection{A particular case of the Ishizaki-Ogawara's theorem}
\label{subsec:ishizaki}

In the case of $q$-difference equations we give a Galoisian proof of the following statement, which is a particular case of
Ogawara's theorem \cite[Theorem 2]{Ogawara-formalIshizaki}. As already noted by Ogawara,
Ishizaki's theorem \cite[Theorem 1.2]{Ishizaki-hypertransc} can be deduced from his formal result.
The latter is proved using elementary complex analysis and it
is a crucial ingredient of \cite{Dreyfus-Hardouin-Roques-Singer-genus0}.
Both Ishizaki's and Ogawara's results are based on the idea that $q$-difference equations
``do not have many solutions which are meromorphic in a neighborhood of $0$''.
In this subsection we only need to assume that $q\neq 0$ is not a root of unity, hence we allow $q$ to have norm equal to $1$.

\begin{prop}\emph{\cite[Theorem 2]{Ogawara-formalIshizaki}}
\label{prop:Ogawara}
Let $q\in\C\smallsetminus\{0,\hbox{roots of unity}\}$, $f\in\C(x)$, $f\neq 0$ and let $z\in \C((x))$ be a formal power series solution of $\tau(z)=z+f$.
The following assertions are equivalent:
\begin{enumerate}
\item $z\in\C(x)$.
\item $z$ is algebraic over $\C(x)$.
\item $z$ is $D$-finite over $\C(x)$.
\item $z$ is differentially algebraic over $\C(x)$.
\end{enumerate}
\end{prop}

\begin{proof}
The implications $1\Rightarrow 2\Rightarrow 3\Rightarrow 4$ are trivial.
We prove that $4\Rightarrow 1$.
We decompose $f(x)$ into elementary fractions and we take care of each part of the decomposition separately.
We consider a pole $\a\in\C^\ast$ of $f(x)$ such that all the other poles of $f(x)$ in $\a q^\Z$ are of the form $q^{-n}\a$, with $n\geq 0$.
Let $N_\a$ the largest integer such that $q^{-N_\a}\a$ is a pole of $f(x)$ and $\sum_{i}\frac{a_i}{(x-q^{-N_\a}\a)^i}$ be the polar part of $f(x)$ at $q^{-N_\a}\a$.
We set $h_1(x)=\sum_{i} \frac{q^ia_i}{(x-q^{-N_\a}\a)^i}$. Replacing $z(x)$ with $z_1(x)=z(x)+h_1(x)$, we are reduced to consider a new functional equation
$y(qx)=y(x)+f_1(x) $, with  $f_1(x)=f(x)-h_1(qx)+h_1(x)$, which has a smaller $N_\a$.
Iterating the argument we obtain a $q$-difference equation with an inhomogeneous term $f(x)$ having at most a single pole in each $q$-orbit $\a q^\Z$.
Corollary \ref{cor:DiffAlgGa} (for $F=\C((x))$ and $K=\C(x)$) implies that there exist $n\geq 0$, $\la_0,\dots,\la_n\in \C$, not all zero, and $g\in\C(x)$ such that
$\la_0f+\la_1\partial(f)+\dots+\la_n \partial^n(f)=\tau(g)-g$.
Since $\tau(g)-g$ cannot have a single pole in $q^\Z\a$, for $\a\neq 0$, we conclude that the rational function
$f(x)$ must have no pole at all in $q^\Z\a$.
We are reduced to prove the claim in the case $f\in\C[x,x^{-1}]$, but this assumption obliges $z(x)\in\C((x))$ to be an element of $\C[x,x^{-1}]$,
as one can see directly from the equation $f(x)=z(qx)-z(x)$, identifying the coefficients of $x^n$, for every integer $n$.
This ends the proof.
\end{proof}

The expansion at zero defines an injective morphism from the field of meromorphic functions over $\C$ to $\C((x))$,
which commutes to the action of $\partial$, therefore we obtain:

\begin{cor}\emph{\cite[Theorem 1.2]{Ishizaki-hypertransc}}
Let $f\in\C(x)$, $f\neq 0$ and let $z\in F$ be a meromorphic function over $\C$, solution of $\tau(z)=z+f$.
Then the assertions of Proposition \ref{prop:Ogawara} are equivalent for $z$.
\end{cor}

\begin{rmk}
The reader can find a Galoisian proof of the Ishizaki theorem in whole generality in \cite[Proposition 3.5]{CharlotteSinger:deltadeltapi},
i.e., for equations of the form $\tau(y)=ay+f$. The general statement can be proven using
the parameterized Galois theory of difference equations.
\end{rmk}

\begin{rmk}
We make some comments on the relation between convergent and meromorphic solutions, under the assumption that $|q|\neq 1$:
\begin{enumerate}
  \item An important property of $q$-difference equations is the following:
\begin{quote}
{\it For a solution of a linear $q$-difference equation with meromorphic coefficients over $\C^\ast$, it is equivalent to be meromorphic in a neighborhood of
zero and to be meromorphic over $\C$.}
\end{quote}
The proof is quite easy and relies on the fact that we have supposed that $|q|\neq 1$. In fact, this allows to consider a meromorphic continuation
of the solution thanks to the fact that any point can be ``brought next to zero'' with a repeated application of $\tau$
or of $\tau^{-1}$. It seems that this remark is originally due to H. Poincaré \cite[page 318]{poincare-1890}.
  \item Let us suppose that $z$ is a meromorphic function over $\C$ and algebraic over $\C(x)$. Then $z$ is a meromorphic function over $\C$, which has at worst a pole at $\infty$,
hence it is rational. This proves that $2\Rightarrow 1$ in Proposition \ref{prop:Ogawara} is true for all linear $q$-difference equations with rational coefficients,
as soon as the solution is meromorphic at $0$.
\end{enumerate}
\end{rmk}


\section{The Galois correspondence (second part)}
\label{sec:GaloisCorrespondenceNormal}

In this section we are going to focus on the role of normal subgroups in the Galois correspondence, under the following assumption.

\begin{assumption}
We suppose that $\tau$ is an automorphism of $F$ and induces an automorphism of $K$.
In difference algebra, when $\tau$ is an automorphism, i.e. admits an inverse, is usually called inversive.
\end{assumption}

The assumption above immediately implies that $\tau$ is also an automorphism of any \PV\ ring and any \PV\ field contained in $F$.
In fact, if $U$ is a fundamental
solution of a system $\tau(Y) = AY$ as in \eqref{eq:system}, we also have $\tau^{-1}(U)=\tau^{-1}(A^{-1})U$.
Notice that we continue to work under the assumption of \S\ref{sec:PVrings} and in particular
that $F$ contains a fundamental solution of the linear system $\tau(Y) = AY$ with coefficients in $K$,
and hence, that all our \PV\ ring and extensions are contained in $F$.

The main result of this section is the following:

\begin{thm}\label{thm:GaloisCorrespondence2}
In the notation of Theorem \ref{thm:GaloisCorrespondence1} above,
let $H$ be an algebraic subgroup of $G$ defined over $k$ and
let $M=L^H$. The following assertions are equivalent:
\begin{enumerate}
  \item $H$ is a normal subgroup of $G$;
  \item $M$ is a \PV\ field over $K$ (for a convenient linear difference equation).
\end{enumerate}
Assuming the equivalent conditions above, the algebraic group
$\Gal(M/K)$ is naturally isomorphic to $G/H$.
\end{thm}

In order to complete the proof of the Galois correspondence,
we need to prove a quite classical proposition on the action of the Galois group on the elements of the Picard-Vessiot extension.
The proof is not difficult and indeed it is quite similar to the differential case \cite[Corollary 1.38]{vanderPutSinger:DifferentialGaloisTheory}, but,
to the best of my knowledge, it is not detailed anywhere in the literature.
Notice that the hypothesis that $\tau$ is inversive is a central ingredient.

\begin{prop}\label{prop:elements of PV}
Let $R\subset F$ be the \PV\ ring for a linear difference system of the form
\eqref{eq:system} over $K$ and $f$ an element of the field of fractions of $R$.
The following statements are equivalent:
\begin{enumerate}
\item
$f\in R$;
\item
the $K$-vector space spanned by $\{\tau^n(f), n\geq 0\}$ has finite dimension.
\end{enumerate}
\end{prop}

\begin{proof}
Let us prove that $(1)\Rightarrow (2)$.
We remind that there exists a fundamental solution
matrix $U$ of a difference system of the form \eqref{eq:system}, such that
$R=K[U,\det U^{-1}]$.
Let us denote by $t_1,\dots, t_{d^2+1}$ the elements of the matrix $U$, plus $\det U^{-1}$.
Since $\tau(U)=AU$ and $\tau(\det U^{-1})=\det A^{-1}\cdot \det U^{-1}$, for any integer $r\geq 1$,
the $K$-vector space generated by the monomials of degree $r$ in the $t_i$'s and their $\tau$-iterated has finite dimension over $K$.
This proves the statement, because any $f\in R$ can be written as a polynomial in the $t_i$'s and hence the
$K$-vector space spanned by $\{\tau^n(f), n\geq 0\}$ is contained in a finite dimensional $K$-vector space.
\par
We now show that $(2)\Rightarrow (1)$.   Let $W$ be the $K$-vector space generated by $\{\tau^n(f), n\geq 0\}$.
We consider the ideal of $R$ defined by $I:=\{a\in R\vert aW\subset R\}$.
Since $f\in L$ and $L$ is the field of fractions of $R$, the ideal $I$ is non-zero.
Moreover $\tau$ is inversive, hence $W\subset\tau^{-1}W$.
Since $W$ and $\tau^{-1}(W)$ are vector spaces of the same dimension, this implies that $\tau^{-1}(W)=W$.
We conclude that $\tau(a)W\subset \tau(a W)\subset R$ for any $a\in I$
and therefore that $I$ is $\tau$-invariant.
Finally, $1\in I$, because $R$ is $\tau$-simple, and  $f\in W\subset R$.
\end{proof}

\begin{cor}
Let $L/K$ be a \PV\ extension and $R$ be the \PV\ ring of $L$.
\begin{enumerate}
  \item
  Let $M$ be an intermediate field which is itself a \PV\ field over $K$.
  Moreover let $R_M$ be its \PV\ ring. Then $R_M=M\cap R$.
  \item
  We fix $f\in R$ and a linear difference equation $\cL(y)=0$ with coefficients in $K$ such that $\cL(f)=0$.
  Furthermore, we suppose that the operator associated with the equation $\cL(y)=0$ has  minimal order $m$
  in $\tau$.
  Then the solutions of $\cL(y)=0$ in $R$ form a $k$-vector space of solutions of maximal dimension $m$.
\end{enumerate}
\end{cor}

\begin{proof}
1. The statement follows from the previous proposition, since $f\in M\cap R$ if and only if $f\in M$ and $f$ is a solution of a linear difference equation
with coefficients in $K$.

\par\noindent
2.
Let $W$ be the space of solution of $\cL(y)=0$ in $R$. Since $f\in W$, we know that $W\neq 0$ and
we can consider a $k$-basis $w_1,\dots, w_r$ of $W$.
The following formula
$$
\det
\left(\begin{array}{ccccc}
w_1 & w_2  & \cdots &  w_r & y \\
\tau (w_1) & \tau (w_2) & \cdots & \tau (w_r)& \tau(y)\\
\vdots & \vdots & \ddots & \vdots & \vdots \\
\tau^{r}(w_1)& \tau^{r}(w_2) & \cdots & \tau^{r}(w_r) & \tau^r(y)
\end{array}\right)=0
$$
gives a $\tau$-difference equation $\wtilde\cL(y)=0$ with coefficients in $L$ having $W$ as space of solutions.
By definition of the Galois group, $\varphi(W\otimes_k B)=W\otimes_k B$, for any $\varphi\in G(B)$ and any $k$-algebra $B$.
The Galois correspondence and the invariance by the action of the Galois group show that the coefficients of $\wtilde\cL(y)=0$ are actually in $K$.
Because of the minimality of the order of the operator associated with $\cL(y)=0$, we conclude that $\cL$ and $\wtilde\cL$ coincide up to the multiplication of a non-zero
element of $K$. This implies that $W\subset R$ has maximal dimension $m$ over $k$.\referee{I have slightly simplified the presentation of the proof.}
\end{proof}

\begin{rmk}\label{rmk:PV-equations}
Notice that in the proof of the second statement above, we could replace $R$ by any $\tau$-$K$-algebra $\wtilde R\subset R$,
such that for any $k$-algebra $B$ and any $\psi\in G(B)$, we have $\psi(\wtilde R\otimes B)\subset (\wtilde R\otimes B)$.
\end{rmk}

\begin{proof}[Proof of Theorem \ref{thm:GaloisCorrespondence2}.]
Let $M$ be a \PV\ field. Then $R_M:=R\cap M$ is a \PV\ ring, which is generated by the entries of a matrix $U$ solution of a
difference linear system with coefficients in $K$, and its inverse.
By definition of the difference Galois group, for any $k$-algebra $B$ and any $\psi\in G(B)$ we have
$\psi(R_M\otimes B)\subset  R_M\otimes B$.
It implies that we have a natural group morphism $G(B)\to \Gal (M/K)(B)$, given by the restriction of the morphisms.
The kernel coincides with $H(B)$, hence $H$ is a normal subgroup of $G$.
\par
Let us suppose that $H$ is a normal subgroup of $G$. We set $M=L^H$ and $R_M=R\cap M$,
so that any $\varphi\in H(B)$ induces the identity over $R_M\otimes B$.
Because of the normality of $H(B)$ in $G(B)$, any $\psi\in G(B)$ verifies $\psi(R_M\otimes B)\subset  R_M\otimes B$.
Finally Remark \ref{rmk:PV-equations} shows that $R_M$ is generated by the solution of a linear difference equations, and hence that it is
a \PV\ ring. We deduce that $M$ is the field of fraction of $R_M$ because
they both coincide with $L^H$.
\end{proof}

\begin{rmk}
In the notation of the proof, for any $k$-algebra $B$, we have a functorial isomorphism\referee{R2. I have rectified the isomorphism.}
$$
(G/H)(B)\cong{\rm Aut}^\tau(R_M\otimes_k B/K\otimes_k B),
$$
which is actually an isomorphism of algebraic group.
\end{rmk}

We remind that $k_n=K^{\tau^n}$. See Example \ref{exa:basic-example-difference-field}.\referee{R1. I have recalled the definition of $k_n$ and
add a reference to the place where it first appears.}

\begin{cor}\label{cor:unitycomponent}
Let $L$ and $G$ be as above and let $G^\circ$ be the connected component of the identity of $G$.
Then, $L^{G^\circ}$ is the relative algebraic closure of $K$ in $L$
(and, hence in our framework coincides with $K(k_n)$, for a convenient positive integer $n$).
\end{cor}

\begin{proof}
Since $G^\circ$ is a normal subgroup of $G$, the finite quotient $G/G^\circ$ is isomorphic to
$\Gal(L^{G^\circ}/K)$. Theorem \ref{thm:dimension} implies that $L^{G^\circ}/K$ is an algebraic extensions, which is also finitely generated.
\par
Let $\wtilde L$ be the relative algebraic closure of $K$ in $L$. Then $L^{G^\circ}\subset \wtilde L$.
Since any algebraic element of $L$ over $K$ is solution of a differential equation over $K$, $\wtilde L$ is also a Picard-Vessiot field,
that therefore correspond to an algebraic subgroup $H$ of $G$, in the sense that $\wtilde L=L^H$.
The inclusion $L^{G^\circ}\subset \wtilde L$ implies that $H\subset G^\circ$. Moreover $G/H$ is a finite group, because it must have dimension $0$, after Theorem
\ref{thm:dimension}.
Since $G^\circ$ is the smallest group such that the quotient $G/G^\circ$ is finite,
we deduce that $H=G^\circ$ and therefore, from the Galois correspondence, that $\wtilde L=L^{G^\circ}$.
\referee{R1. I have developed the last sentence}
\end{proof}

\appendix
\section{Behavior of the Galois group with respect to the iteration of $\tau$}

Let us consider the system \eqref{eq:system} and its $n$-th iteration:
\beq\label{eq:dsystem}
\tau^ny=A_ny,
\hbox{~where~} A_n:=\tau^{n-1}(A)\cdots\tau(A)A.
\eeq
We want to compare the Galois group of \eqref{eq:system} with the Galois group of \eqref{eq:dsystem}.
\par
It follows from the Definition \ref{defn:PV} of Picard-Vessiot ring and field \referee{R1. I have referred more precisely to the definition.} that,
if $R$ (resp. $L$) is a \PV\ ring (resp. field) for \eqref{eq:system} over $K$,
$R(k_n)$   (resp. $L(k_n)$) is also a \PV\ ring (resp. field) for \eqref{eq:dsystem} over $K(k_n)$.
Let $G_n:=\Gal^{\tau^n}(L(k_n)/K(k_n))$, where we have add the superscript $\tau^n$ to the notation with the obvious meaning, to avoid any confusion.

\medskip
Let $G^\circ _1$ be the identity component of $G_1$.
By Corollary \ref{cor:unitycomponent}, there exists $r$, such that $k_r\subset L$ and that
$\Gal^\tau(L/K(k_r))=G_1^\circ$.

\begin{lemma}
In the notation above, we have $\Gal^{\tau^r}(L/K(k_r))=G_1^\circ\otimes_k k_r$.
\end{lemma}

\begin{proof}
By definition, for any $k_r$- algebra $B$, we have an injective morphisms from
$G^\circ_1(B)\to  \Gal^{\tau^r}(L/K(k_r))(B)$, indeed if a morphisms commutes with $\tau$, it commutes also with $\tau^n$.
The equality follows from the fact that the groups are connected and that they have the same dimension,
by Theorem \ref{thm:dimension}.\referee{R1. I have explain the link with the theorem.}
\end{proof}

 \begin{prop}
 In the notation introduced above,
for any $n\geq r$,
 the Galois group of \eqref{eq:system} over $K(k_n)$ is isomorphic to
 the Galois group of \eqref{eq:dsystem} over $K(k_n)$.
 \end{prop}

\begin{proof}
The Galois group of \eqref{eq:system} over $K(k_n)$ is $\Gal^\tau(L(k_n)/K(k_n))\cong G_1^\circ \otimes_k k_n$.
It can be naturally seen as a subgroup of  $\Gal^{\tau^n}(L(k_n)/K(k_n))$.
Equality follows from Theorem \ref{thm:dimension} and the connectedness of the groups, as in the proof of the previous
lemma.\referee{R1. Same argument as in the previous proposition.}
\end{proof}

\providecommand{\bysame}{\leavevmode\hbox to3em{\hrulefill}\thinspace}
\providecommand{\MR}{\relax\ifhmode\unskip\space\fi MR }
\providecommand{\MRhref}[2]{%
  \href{http://www.ams.org/mathscinet-getitem?mr=#1}{#2}
}
\providecommand{\href}[2]{#2}

\end{document}